\documentclass{llncs}
\usepackage[utf8]{inputenc}
\usepackage{hyperref}
\usepackage{verbatim}
\usepackage{amsmath}
\usepackage{amssymb}
\usepackage{mathtools}
\usepackage{caption,color}
\def\ind{\vbox{\hbox{$\bot$\kern-.6em$\bot$}\kern-.05em}}
\def\dep{\vbox{\hbox{$\top$\kern-.6em$\top$}\kern-.05em}}
\newcommand*{\normally}{\mathrel{\ooalign{$|$\hfil\cr\kern+1pt$\thicksim$}}} 
\newcommand*{\nnormally}{\mathrel{\ooalign{$|$\hfil\cr\kern+1pt$\thicksim$}\negthickspace \negthickspace /} } 

\def\A{\mathcal{A}}

\def\I{\mathcal{I}}

\def\F{\mathcal{F}}
\def\P{\mathcal{P}}

\def\J{\mathcal{J}}
\def\proj{\rho}

\title{
Transitive reasoning  with imprecise probabilities
}
\author{Angelo
  Gilio\inst{1} \and Niki Pfeifer\inst{2} \and  Giuseppe Sanfilippo\inst{3}}
\institute{
Department SBAI,
  University of Rome ``La Sapienza'', Italy
  \and 
   Munich Center for Mathematical Philosophy, LMU Munich, Germany
  \and
  Department of Mathematics and Computer Science,
University of Palermo, Italy  
}
\begin{document}
\maketitle
\begin{abstract}
We study probabilistically informative (weak) versions of transitivity,  by using suitable definitions of defaults and negated defaults, in the setting of coherence and imprecise probabilities. We represent  {$\mbox{p-consistent}$}  sequences  of defaults and/or negated defaults by  g-coherent  imprecise probability assessments on the respective sequences of conditional events. Finally, we prove the coherent probability propagation rules
  for Weak Transitivity and  the validity of selected
  inference patterns by proving the p-entailment for the associated
  knowledge bases.
\end{abstract}
\section{Motivation and outline}
While Transitivity is basic for reasoning, it does not hold in nonmonotonic reasoning systems. Therefore,
various patterns of Weak Transitivity were studied in the  literature (e.g., \cite{freund1991}).
In probabilistic
approaches, Transitivity is probabilistically non-informative, i.e.,
the premise probabilities, $p(C|B), p(B|A)$, do not constrain the
probability of the conclusion $p(C|A)$ (for instance, the extension $p(C|A)=z$ of the assessment  $p(C|B)=1, p(B|A)=1$   is coherent for any $z\in[0,1]$; see 
\cite{pfeifer09b,pfeifer10a}). 
In this paper, we study
probabilistically informative versions of Transitivity in the setting
of coherence (\cite{biazzo00,coletti02,gilio13}). Transitivity has also been  studied in \cite{bonnefon12,dubois93}; among other differences, in our approach we use imprecise probabilities in the setting of coherence, where conditioning events may have zero probability.\\ After introducing some  notions of coherence for set-valued probability assessments  
(Sect.~\ref{SEC:IP}), we present probabilistic interpretations of
defaults and negated defaults (Sect.~\ref{SEC:ProbKBENT}). We represent a sequence of  defaults and/or negated defaults (knowledge base) by an imprecise probability assessment on the associated 
 sequence of conditional events. Moreover, we generalize  definitions of p-consistency and p-entailment. In
Sect.~\ref{SEC:WT} we prove the coherent probability propagation
rules for Weak Transitivity (Theorem~\ref{THM:PROPWT}). We then
exploit Theorem~\ref{THM:PROPWT} to demonstrate the validity of
selected patterns of (weak) transitive inferences involving 
defaults and negated defaults by proving the p-entailment for the corresponding knowledge bases
(Sect.~\ref{SEC:WTIND}). 
\section{Imprecise probability assessments} 
\label{SEC:IP}
Given two events $E$ and  $H$, with $H\neq \bot$,  the \emph{conditional event} $E|H$ is defined as a three-valued logical entity which is \emph{true} if $EH$ (i.e., $E\wedge H$) is true, 
\emph{false} if $\neg EH$ is true, and  \emph{void} if $H$ is false. Given a finite sequence of $n\geq 1$ conditional events $\F=(E_1|H_1,\ldots ,E_n|H_n)$, we denote  by $\P$ \emph{any precise} probability assessment $\mathcal{P}=(p_1,\ldots,p_n)$
on  $\F$, where  $p_j=p(E_j|H_j)\in[0,1]$,  $j=1,\dots,n$. 
Moreover, we denote by $\Pi$ the set of \emph{all  coherent precise} assessments on $\F$. The coherence-based probabilistic approach has been adopted by many authors (see e.g., \cite{biazzo00,coletti02,gilio02,gilio13ins,GiSa14,pfeifer13,pfeifer09b}); in this paper, due to the lack of space, we do not insert the basic notions on coherence. 
We recall that when there are no logical relations among the events $E_1,H_1,\ldots, E_n,H_n$ involved in $\F$, that is   $E_1,H_1,\ldots, E_n,H_n$ are logically independent, then the set $\Pi$ associated with $\F$ is the whole unit hypercube $[0,1]^n$. 
If there are  logical relations, then  the set $\Pi$ \emph{could be} a strict subset of  $[0,1]^n$. As is well known $\Pi\neq \emptyset$; therefore, 
$\emptyset\neq \Pi\subseteq [0,1]^n$.
\begin{definition}\label{DEF:IA}
An \emph{imprecise, or set-valued, assessment} $\I$ on a  family of  conditional events   $\F$  is a (possibly empty) set of precise assessments $\P$ on $\F$.
\end{definition}
Definition \ref{DEF:IA},  introduced in  \cite{gilio98}, states that  an \emph{imprecise (probability) assessment} $\I$ on a given family  $\F$ of $n$ conditional events is just a (possibly empty)  subset of $[0,1]^n$.
Given an imprecise assessment $\I$ we denote by $\I^c$ the \emph{complementary imprecise assessement} of $\I$, i.e.  $\I^c=[0,1]^n\setminus \I$.

\begin{definition}\label{DEF:GCOH}
Let be given a sequence  of $n$ conditional events  $\F$. An imprecise assessment $\I\subseteq [0,1]^n$  on $\F$ is  \emph{g-coherent}  if and only if there exists  a coherent precise assessment  $\P$  on $\F$ such that $\P\in \I$.
\end{definition}
\begin{definition} 
Let  $\I$ be a  subset of $[0,1]^n$.  For each $j \in\{1,2,\ldots,n\}$, the projection $\proj_j(\I)$ of $\I$  onto the $j$-th coordinate,  is defined as 
\[
\proj_j(\I)=\{x_j\in [0,1]: p_j=x_j, \mbox{ for some }(p_1,\ldots,p_n)\in \I\}.
\]
\end{definition}
\begin{definition}
An imprecise assessment $\I$ on a sequence of $n$ conditionals event  $\F$  is  \emph{coherent}  if and only if,  for every  $j \in \{1,\ldots,n\}$ and for every  $x_j\in \proj_j(\I)$, there  exists a coherent precise assessment  $\P=(p_1,\ldots,p_n)$ on $\F$, such that $\P \in\I$ and  $p_j= x_j$.
\end{definition}
\begin{definition}\cite[Definition 2]{gilio98}
An imprecise assessment $\I$  on $\F$  is  \emph{totally coherent}  if and only if the following two conditions are satisfied: (i) $\I$ is non-empty;   (ii) if $\P\in \I$, then $\P$ is a coherent precise assessment on $\F$. 
\end{definition}

\begin{remark}
We observe that:
\[
\begin{array}{l}\label{REM:GTCOHERENCE}
\I \mbox{ is g-coherent }\; \Longleftrightarrow \; \Pi \cap \I \neq \emptyset
\; \Longleftrightarrow\;  \proj_j(\Pi \cap \I ) \neq \emptyset  , \; \forall j=1,\ldots,n\,;
\\
\I \mbox{ is coherent } \;\Longleftrightarrow \; \emptyset  \neq\proj_j(\Pi \cap \I )= \proj_j(\I), \; \forall j=1,\ldots,n\,;
\\
\I \mbox{ is totally coherent } \;\Longleftrightarrow \; \emptyset \neq  \Pi \cap \I =\I\,.
\end{array}
\]
Then, the following relations among the different notions of \emph{coherence} hold:\\
$\I$  totally coherent $\Rightarrow  \I $ coherent $\Rightarrow \I$  g-coherent\,.
\end{remark}

\begin{definition} 
Let  $\I$ be a non-empty subset of $[0,1]^n$.  For each sub-vector  $(j_1,j_2,\ldots,j_m)$ of $(1,2,\ldots,n)$, the projection $\proj_{(j_1,j_2,\ldots,j_m)}(\I)$ of $\I$
 onto the  coordinates  $(j_1,\ldots,j_m)$, with $1\leq m\leq n$, is defined as 
\[
\begin{array}{ll}
\proj_{(j_1,\ldots,j_m)}(\I)=\\
\{(x_{j_1},\ldots, x_{j_m})\in [0,1]^m: p_j=x_j, j=j_1,\ldots,j_m, \mbox{ for some }  (p_1,\ldots,p_n)\in \I\, \}.
\end{array}
\]
\end{definition}
\noindent Let $\I$ be an imprecise  assessment on the sequence  $\F=(E_1|H_1,E_2|H_2,\ldots ,E_n|H_n)$; moreover, let $E_{n+1}|H_{n+1}$ be a further conditional event and let $\J\subseteq [0,1]^{n+1}$ an imprecise assessment on $(\F ,E_{n+1}|H_{n+1})$. We say that $\J$ is an \emph{extension} of $\I$ to $(\F ,E_{n+1}|H_{n+1})$ iff $\proj_{(1,2,\ldots,n)}(\J)=\I$, that is:  
$(i)$ for every $(p_1,\ldots,p_n,p_{n+1})\in \J$, it holds that $(p_1,\ldots,p_n)\in \I$; $(ii)$ for every $(p_1,\ldots,p_n)\in \I$, there exists $p_{n+1}\in [0,1]$ such that  $(p_1,\ldots,p_n,p_{n+1})\in \J$.

\begin{definition}\label{DEF:GCOHEXT}
Let $\I$ be a g-coherent assessment on   $\F=(E_1|H_1,E_2|H_2,\ldots ,E_n|H_n)$; moreover, let $E_{n+1}|H_{n+1}$ be a further conditional event and let $\J$ be an extension of $\I$  to $(\F ,E_{n+1}|H_{n+1})$. We say that $\J$  is a \emph{g-coherent extension} of $\I$ if and only if $\J$ is g-coherent.
\end{definition}
\begin{theorem}
Given a g-coherent assessment $\I\subseteq [0,1]^n$ on $\F$, let  $E_{n+1}|H_{n+1}$ be a further conditional event.
Then, there exists a g-coherent extension $\J\subseteq [0,1]^{n+1}$  of $\I$ to the family  $(\F, E_{n+1}|H_{n+1})$.
\end{theorem}
\begin{proof}
As $\I$ is g-coherent,  there exists a  coherent precise assessment $\P=(p_1,\ldots,p_n)$ on $\F$, with $\P\in \I$. Then, as is well known (FTP), there exists (a non-empty interval) $[p',p'']\subseteq [0,1]$ such that $(\P,p_{n+1})$ is a  coherent precise assessment on $\F\cup \{E_{n+1}|H_{n+1}\}$, for every  
$p_{n+1}\in [p',p'']$. Now, let be given any $\Gamma\subseteq [0,1]$ such that $\Gamma \cap [p',p'']\neq \emptyset$; moreover, let us consider the extension  $\J=\I \times \Gamma$. Clearly,  $(\P,p_{n+1})\in \J$ for every  $p_{n+1}\in \Gamma \cap [p',p'']$; moreover the assessment $(\P,p_{n+1})$ on  $(\F, E_{n+1}|H_{n+1})$ is  coherent for every  $p_{n+1}\in \Gamma \cap [p',p'']$. Finally, $\J$ is a g-coherent extension of $\I$ to $(\F, E_{n+1}|H_{n+1})$.
\end{proof}
Given a g-coherent assessment $\I$ on  a sequence of $n$ conditional events $\F$, for each coherent precise assessment  $\P$ on $\F$, with $\P\in \I$, we denote by $[\alpha_{\P},\beta_{\P}]$ the interval of coherent extensions of $\P$  to $E_{n+1}|H_{n+1}$; that is, the assessment  $(\P,p_{n+1})$ on $(\F,E_{n+1}|H_{n+1})$ is coherent if and only if
$p_{n+1}\in [\alpha_{\P},\beta_{\P}]$. Then, defining the set 
\begin{equation} \label{EQ:UEXT}
\begin{array}{ll}
\Sigma=\bigcup_{\P\in\Pi \cap \I}[\alpha_{\P},\beta_{\P}]\,,
\end{array}
\end{equation}
for every $p_{n+1}\in \Sigma$, the assessment $\I\times \{p_{n+1}\}$ is a g-coherent extension of $\I$ to $(\F,E_{n+1}|H_{n+1})$; moreover, 
 for every $p_{n+1}\in [0,1]\setminus \Sigma$, the extension $\I\times \{p_{n+1}\}$  of $\I$ to $(\F,E_{n+1}|H_{n+1})$ is  not g-coherent. Thus, denoting by $\Pi'$ the set of coherent precise  assessments on $(\F,E_{n+1}|H_{n+1})$, it holds that  
 $\Sigma$ is the projection onto  the $(n+1)$-th coordinate of the set $(\I\times [0,1]) \cap \Pi'  $,  that is
$\proj_{n+1}((\I\times [0,1]) \cap \Pi')=\Sigma$.
 We say that $\Sigma$ is the \emph{set of coherent extensions}  of the imprecise assessment  $\I$ on $\F$ to the  conditional event $E_{n+1}|H_{n+1}$.
\section{Probabilistic  knowledge bases and entailment}
\label{SEC:ProbKBENT}
The sentence  ``\emph{$E$ is a plausible consequence of  $H$},'' where $E$ is an event and $H$ is a not self-contradictory event, is a \emph{default}, which we denote by ${H\normally E}$. Moreover, we denote a  \emph{negated  default}, $\neg (H \normally E)$, by $H\nnormally E$ (it is not the case, that: $E$ is a plausible consequence of  $H$). We denote by $s$ a default  or a negated default, which are  defined in terms of probabilistic assessments as follows:
\begin{definition}
${H\normally E}$ holds iff $p(E|H)=1$.   ${H\nnormally E}$ holds iff $p(E|H)\neq 1$.
\end{definition}
Given two events $E$ and  $H$, with $H\neq \bot$, by  coherence $p(E|H)+p(\neg E|H)=1$ (which holds in general). 
Thus, the probabilistic interpretation of the following types of sentences
${H\normally E}$ (I),  ${H\normally \neg E}$ (II), ${H\nnormally E}$ (III), and  ${H\nnormally \neg E}$ (IV), can be represented in terms of imprecise assessments on $E|H$ (Table \ref{Table:BS}).
\begin{table}[!ht]
\vspace{-1em}
\centering %
\begin{tabular}{cccc}\hline 
 Type \quad & \quad Sentence   \quad  & \quad   Probabilistic interpretation \quad  & \quad     Assessment $\mathcal{I}$ on $E|H$ \\
\hline
I &$H\normally E$&  $p(E|H)=1$  & $\{1\}$  \\
II & $H\normally \neg E$&  $p(\neg E|H)=1$ & $\{0\}$\\
III & $H\nnormally E$&  $p(E|H)<1$  & $[0,1[$\\
IV & $H\nnormally \neg E$&  $p(\neg E|H)<1$  & $]0,1]$\\
 \hline 
\end{tabular}
\caption{Probabilistic interpretations of defaults  (types I and II)  and negated defaults  (types III and IV), and their respective (imprecise) assessments $\mathcal{I}$ on a conditional event $E|H$.}
\label{Table:BS} %
\vspace{-2em}
\end{table}
In this paper a knowledge base    $\mathcal{K}$ is defined as a (non-empty) finite sequence of defaults and negated defaults. 
Let 
${\mathcal{K}=({H_1 \normally E_1},\ldots,{H_n \normally E_n},{D_1 \nnormally C_1},\ldots, 
{D_m \nnormally C_m})}$ be a knowledge base, with $n+m\geq 1$. The pair $(\mathcal{F}_{\mathcal{K}}, \mathcal{I}_{\mathcal{K}})$, where
 $\mathcal{F}_{\mathcal{K}}$ is the ordered family of conditional events $(E_1|H_1,\ldots,E_n|H_n,C_1|D_1,\ldots,C_m|D_m)$
and  
$\mathcal{I}_{\mathcal{K}}$ is the imprecise assessment
 ${\bigtimes_{i=1}^{n}\{1\}\times \bigtimes_{j=1}^{m}[0,1[}$
on $\mathcal{F}_{\mathcal{K}}$.
Thus, we define our probabilistic representation  of  the knowledge base $\mathcal{K}$ by the corresponding pair $(\mathcal{F}_{\mathcal{K}}, \mathcal{I}_{\mathcal{K}})$. We now define the notion of p-consistency of a given  knowledge  base
 in terms of g-coherence.
\begin{definition}\label{DEF:PCONS}
A knowledge base $\mathcal{K}$ is 
\emph{p-consistent} if and only if the imprecise assessment $\mathcal{I}_{\mathcal{K}}$ on  $\mathcal{F}_{\mathcal{K}}$  is g-coherent. 
\end{definition}
By coherence,  any (non-empty) sub-sequence $\mathcal{S}$ of a p-consistent knowledge base $\mathcal{K}$ is also a p-consistent knowledge base.
We define the notion of p-entailment of a (negated) default from a  p-consistent  knowledge  base in terms of coherent extension of a g-coherent  assessment.
\begin{definition}\label{DEF:ENTAIL}
Let $\mathcal{K}$ be p-consistent.
$\mathcal{K}$  {\em p-entails}  ${A\normally B}$ (resp., $A\nnormally B$), denoted by 
$\mathcal{K}\models_p A\normally B$ (resp., $\mathcal{K}\models_p A\nnormally B$),  iff 
the (non-empty) set of coherent extensions  to $B|A$ of   $\mathcal{I}_{\mathcal{K}}$ on  $\mathcal{F}_{\mathcal{K}}$ is  $\{1\}$ (resp., a subset of $[0,1[$\;).
\end{definition}
\begin{theorem} Let $\mathcal{K}$ be p-consistent.
$\mathcal{K}\models_p {A\normally B}$ (resp., $\mathcal{K}\models_p {A\nnormally B})$,   iff
there exists a (non-empty) sub-sequence  $\mathcal{S}$ of $\mathcal{K}$: 
$\mathcal{S}\models_p {A\normally B}$ (resp., $\mathcal{S}\models_p {A\nnormally B})$.
\end{theorem}
\begin{proof}
($\Rightarrow$) Trivially, by setting $\mathcal{S}=\mathcal{K}$.\\ 
($\Leftarrow$) Assume that $\mathcal{S}\models_p A\normally B$ (resp., $A\nnormally B$).  Then,  for every precise coherent assessment
$\mathcal{P}\in \mathcal{I}_{\mathcal{S}}$ on $\mathcal{F}_{\mathcal{S}}$, if the  extension $(\mathcal{P},z)$ on $(\mathcal{F}_{\mathcal{S}},B|A)$ is coherent, 
then  ${z=1}$  (resp., $z\neq 1$).
Let $\mathcal{P}'\in \mathcal{I}_\mathcal{K}$ be a coherent precise assessment on $\mathcal{F}_\mathcal{K}$. 
For \emph{reductio ad absurdum} we assume   that the extension $(\mathcal{P}',z)$ on $(\mathcal{F}_{\mathcal{K}},B|A)$ is  coherent with $z\in[0,1[$ (resp., $z=1$). Then, the sub-assessment $(\mathcal{P},z)$ of $(\mathcal{P}',z)$ on  $(\mathcal{F}_{\mathcal{S}},B|A)$  is coherent with $z\in[0,1[$ (resp., $z=1$): this  contradicts  ${\mathcal{S}\models_p A\normally B}$ (resp., ${\mathcal{S}\models_p A\nnormally B}$). Therefore, ${\mathcal{K}\models_p A\normally B}$ (resp., ${\mathcal{K}\models_p A\nnormally B}$).\qed
\end{proof}
A similar approach has been  developed in \cite[Definition 26]{coletti02}. We observe that if the knowledge base  $\mathcal{K}$ consists of defaults only, then  definitions \ref{DEF:PCONS} and \ref{DEF:ENTAIL} coincide with the notion of p-consistency and p-entailment, respectively, investigated from a coherence perspective in \cite{gilio13} 
(see also \cite{biazzo02,gilio13ins}). Moreover, p-entailment of the well known inference rules of the nonmonotonic System P has been studied in this context (e.g., \cite{coletti02,gilio02}, see also \cite{benferhat97,CoPV14}).
%
%
%
\begin{remark}\label{REM:CONJUGACY}
By Table \ref{Table:BS}  the probabilistic interpretation of $\mathcal{K}=({H_1 \normally E_1},\ldots,{H_n \normally E_n},$ ${D_1 \nnormally C_1},\ldots,
{D_m \nnormally C_m})$   can equivalently be represented by  $\mathcal{I}_{\mathcal{K}}={\bigtimes_{i=1}^{n}\{1\}\times \bigtimes_{j=1}^{m}]0,1]}$
on $\mathcal{F}_{\mathcal{K}}=(E_1|H_1,\ldots,E_n|H_n,$ $\neg C_1|D_1,\ldots,\neg C_m|D_m)$.
Definitions \ref{DEF:PCONS} and \ref{DEF:ENTAIL} can be rewritten accordingly. 
\end{remark}
\begin{example}
Given three logically independent events $A,B,C$, with $A\neq \bot$,  any assessment $(x,y)\in[0,1]^2$ on $(C|A,B|A)$ is of course coherent. Furthermore,   the extension $z=P(C|AB)$ of $(x,y)$ on $(C|A,B|A)$ is coherent if and only if  $z\in [z',z'']$, where (\cite{gilio02})
\begin{equation*}\label{EQ:Cautios}
\begin{small}
\begin{array}{ccc}
z'=\left\{
\begin{array}{ll}
\frac{x+y-1}{y}>0, & \mbox{ if } x+y>1,\\
0, & \mbox{ if }  x+y\leq 1,\\
\end{array}
\right.
&\quad 
z''=\left\{\begin{array}{ll}
\frac{x}{y}<1, & \mbox{ if } x<y,\\
1, & \mbox{ if }   x\geq y\,.
\end{array}
\right.
\end{array}
\end{small}
\end{equation*}
Then,  we have (see also \cite{coletti02,freund1991}):
${(A\normally C, A\normally B)\models_p A B\normally C}$  (Cautious Monotonicity);
${(A\nnormally C, A\nnormally \neg B )\models_p A B\nnormally C}$ (Rational Monotonicity).\\
\end{example}
\section{Weak Transitivity: Propagation of probability bounds }
\label{SEC:WT}
In this section we   compute the interval $[z',z'']$  of the coherent extensions ${z=p(C|A)}$ of any coherent  assessment $(x,y,t)\in[0,1]^3$ on $(C|B,B|A,A|A\vee B)$ to $C|A$, by applying the Algorithm 2 given in \cite{biazzo00}.  The algorithm's outputs $p_0$ and $p^0$ will be denoted by $z'$ and $z''$, respectively. 
\begin{remark}\label{REM:TOTCOHWT}
Let $A,B,C$ be logically independent events.
It can be proved that  the assessment $(x,y,t)$  on $\mathcal{F}=(C|B,B|A,A|A\vee B)$ is coherent for every $(x,y,t)\in [0,1]^3$, that is the imprecise assessment $\mathcal{I}=[0,1]^3$ on $\mathcal{F}$ is totally coherent. 
Also $\mathcal{I}=[0,1]^3$ on $\mathcal{F}'=(C|B,B|A,C|A)$ is totally coherent.\footnote{For proving total coherence of $\mathcal{I}$ on $\mathcal{F}$ (resp., $\mathcal{F}'$)  it is sufficient to check that the assessment $\{0,1\}^3$ on $\mathcal{F}$ (resp.,   $\mathcal{F}'$) is totally coherent (\cite[Theorem 7]{gilio98}), i.e, each of the eight vertices  of the unit cube is coherent. Coherence can be checked, for example, by applying Algorithm 1 of \cite{gilio98} or by the CkC-package \cite{capotorti09}.} 
\end{remark}
\paragraph{Computation of the lower probability bound  $z'$ on $C|A$}~\\
\emph{Input:}  $n=3$, $\F_n=(E_1|H_1,E_2|H_2,E_3|H_3)=(C|B,B|A,A|A\vee B)$, $\A_n=([\alpha_1,\beta_1],[\alpha_2,\beta_2],[\alpha_3,\beta_3])= ([x,x],[y,y],[t,t])$,  $E_{n+1}|H_{n+1}=C|A$.\\
\emph{Step 0.\,} The constituents  associated with  ${(\F_n,E_{n+1}|H_{n+1}})=(C|B,B|A,A|(A\vee B),C|A)$
and contained in $\mathcal{H}_0=\bigvee_{i=1}^{n+1}H_i= A \vee B$ are $C_1=AB C\,, C_2=AB  \neg C\,, C_3=A \neg B  C\,, C_4=A \neg B  \neg C\,, C_5=\neg ABC\,,$ and  $C_6=\neg A B  \neg C$.
We construct the following starting
 system with unknowns $\lambda_1,\ldots,\lambda_6,z=p_{n+1}$:

\begin{equation}\label{S_0'}
 \left\{
\begin{array}{lllllll}
\lambda_1+\lambda_3=z(\lambda_1+\lambda_2+\lambda_3+\lambda_4), \\
\lambda_1+\lambda_5=x(\lambda_1+\lambda_2+\lambda_5+\lambda_6), \\
\lambda_1+\lambda_2=y(\lambda_1+\lambda_2+\lambda_3+\lambda_4),\\
\lambda_1+\lambda_2+\lambda_3+\lambda_4=t(\lambda_1+\lambda_2+\lambda_3+\lambda_4+\lambda_5+\lambda_6), \\
\lambda_1+\lambda_2+\lambda_3+\lambda_4+\lambda_5+\lambda_6=1, \;\;
\lambda_i\geq 0,\;\; i=1,\ldots,6\,.
\end{array}
\right.
\end{equation}
\emph{Step 1.\,} We set $z=0$ in  System (\ref{S_0'}) and  obtain
\begin{equation}\label{S_0'z=0}
\begin{array}{lcl}
\left\{
\begin{array}{lllllll}
\lambda_1+\lambda_3=0, \\
\lambda_5=x(\lambda_2+\lambda_5+\lambda_6), \\
\lambda_2=y(\lambda_2+\lambda_4),\\
\lambda_2+\lambda_4=t, \\
\lambda_2+\lambda_4+\lambda_5+\lambda_6=1, \\
\lambda_i\geq 0,\; i=1,\ldots,6\,;
\end{array}
\right.
&
\quad \Longleftrightarrow \quad &
\left\{
\begin{array}{lllllll}
\lambda_1=\lambda_3=0, \\
\lambda_2=yt,\\
\lambda_4=t(1-y), \\
\lambda_5=x(yt+1-t), \\
\lambda_6=(1-t)(1-x)-xyt, \\
\lambda_i\geq 0,\; i=1,\ldots,6\,.
\end{array}
\right.
\end{array}
\end{equation}
As $(x,y,t) \in [0,1]^3$, we observe that: $yt\geq 0$, $t(1-t)\geq 0$, and  $x(yt+1-t)\geq 0$.  Thus, System (\ref{S_0'z=0})   is solvable  if and only if $xyt\leq (1-t)(1-x) $, that is 
$t(1-x+xy)\leq 1-x$. We distinguish two cases:  $(i)$ $t(1-x+xy)> 1-x$; $(ii)$ $t(1-x+xy)\leq 1-x$.  In Case~$(i)$,  System~(\ref{S_0'z=0})   is not solvable and---according to the algorithm---we proceed to  Step~2. In Case~$(ii)$,  System~(\ref{S_0'z=0})  is solvable and we go to Step~3. \\
\emph{Case $(i)$.} We take Step 2 and consider the following linear programming problem:
\emph{Compute $z'=\min(  \lambda_1+\lambda_3)$ subject to:}
\begin{equation}\label{Sigma'}
\left\{
\begin{array}{lllllll}
\lambda_1+\lambda_5=x(\lambda_1+\lambda_2+\lambda_5+\lambda_6), \;\;
\lambda_1+\lambda_2=y(\lambda_1+\lambda_2+\lambda_3+\lambda_4), \\
\lambda_1+\lambda_2+\lambda_3+\lambda_4=t(\lambda_1+\lambda_2+\lambda_3+\lambda_4+\lambda_5+\lambda_6), \\
\lambda_1+\lambda_2+\lambda_3+\lambda_4=1, \;\;
\lambda_i\geq 0,\; i=1,\ldots,6.
\end{array}
\right.
\end{equation}
As $t(1-x+xy)> 1-x\geq 0$, it holds that  $t>0$. In this case, 
the constraints in (\ref{Sigma'}) can be rewritten in the following way
\[
\begin{array}{lll}
\left\{
\begin{array}{lllllll}
\lambda_1+\lambda_5=x\big(y+\frac{1-t}{t}\big), \\
\lambda_1+\lambda_2=y, \\
\lambda_5+\lambda_6=\frac{1-t}{t}, \\
\lambda_3+\lambda_4=1-y, \\
\lambda_i\geq 0,\; i=1,\ldots,6,
\end{array}
\right.
& \quad \Longleftrightarrow \quad &
\left\{
\begin{array}{lllllll}
\lambda_5=xy+x\frac{1-t}{t}-\lambda_1, \\
\lambda_2=y-\lambda_1, \\
\lambda_6=\frac{1-t}{t}-xy-x\frac{1-t}{t}+\lambda_1,\\
\lambda_4=1-y-\lambda_3,\\
\lambda_i\geq 0,\; i=1,\ldots,6.
\end{array}
\right.
\end{array}
\] 
that is 
\begin{equation}\label{Sigma'1}
\left\{
\begin{array}{lllllll}
\max\big\{0,xy-\frac{(1-t)(1-x)}{t}\big\}\leq \lambda_1 \leq \min\big\{y,xy+x\frac{1-t}{t}\big\}, \\
\lambda_2=y-\lambda_1, \;\;
0\leq\lambda_3\leq 1-y, \;\;
\lambda_4=1-y-\lambda_3,\\
\lambda_5=xy+x\frac{1-t}{t}-\lambda_1, \;\;
\lambda_6=\frac{(1-t)(1-x)}{t}-xy+\lambda_1.
\end{array}
\right.
\end{equation}
As $t(1-x+xy)> 1-x\geq 0$, it holds that $xy-(1-x)(1-t)/t>0$. Thus, 
we obtain the minimum of $(\lambda_1+\lambda_3)$
subject to~(\ref{Sigma'1})  at $(\lambda_1',\lambda_3')={(xy-(1-t)(1-x)/t,0)}$.
The \emph{procedure stops} yielding as \emph{output}
$z'=\lambda_1'+\lambda_3'= xy-(1-t)(1-x)/t>0$.\\
\emph{Case $(ii)$.} We take Step~3 of the algorithm. We denote by $\Lambda$  and $\mathcal{S}$ the vector of unknowns $(\lambda_1,\ldots,\lambda_6)$ and the set of solution of System~(\ref{S_0'z=0}), respectively.  By recalling the conditioning events $H_1=B, H_2=H_4=A, H_3=A \vee B$, we consider the following linear functions and their maxima in $\mathcal{S}$:
\begin{equation}\label{EQ:PHI}
\begin{array}{l}
\Phi_{1}(\Lambda)=\sum_{r:C_r\subseteq B}\lambda_r= \lambda_1+\lambda_2+\lambda_5+\lambda_6, \\ \Phi_{2}(\Lambda)=\Phi_{4}(\Lambda)=\sum_{r:C_r\subseteq A}\lambda_r=\lambda_1+\lambda_2+\lambda_3+\lambda_4, \\ 
\Phi_{3}(\Lambda)=\sum_{r:C_r\subseteq A\vee B}\lambda_r=\lambda_1+\lambda_2+\lambda_3+\lambda_4+\lambda_5+\lambda_6\,, \\
 M_i=\max_{\Lambda \in \mathcal{S}} \Phi_i(\Lambda),\;\;  i=1,2,3,4\;.
\end{array}
\end{equation}
From System~(\ref{S_0'z=0}), we obtain:
$ \Phi_{1}(\Lambda)=yt+1-t$, $\Phi_{2}(\Lambda)=\Phi_{4}(\Lambda)=t$,  and $\Phi_3(\Lambda)=1,\;\; \forall \Lambda \in \mathcal{S}$. Then,  $M_1=yt+1-t$,  $M_2=M_4=t$, and $M_3=1$.
We consider two subcases: $t>0$; $t=0$.
If $t>0$, then $M_{4}>0$ and we are in case 1 of Step 3. Thus,  the \emph{procedure stops} and yields $z'=0$ as \emph{output}.  If $t=0$, then
$M_{1}>0,M_{3}>0$ and $M_{2}=M_{4}=0$. Hence, we are in case 3 of Step 3 with $J = \{2\}, I_0=\{2,4\}$ and the procedure restarts with Step 0, with $\F_n$ replaced by  $\F_J=(E_2|H_2)=(B|A)$ and $\A_n$ replaced by $\A_J=([\alpha_2,\beta_2])=([y,y])$.\\
\emph{(2\textsuperscript{nd}) Step $0$.\,} The constituents  associated with  $(B|A,C|A)$, contained in $A$, are
$ C_1=ABC, C_2=AB\neg C, C_3=A\neg BC, C_4={A\neg B\neg  C}$.
The starting system is
\begin{equation}\label{S_1'}
\left\{
\begin{array}{lllllll}
\lambda_1+\lambda_2=y(\lambda_1+\lambda_2+\lambda_3+\lambda_4), \;\;
\lambda_1+\lambda_3=z(\lambda_1+\lambda_2+\lambda_3+\lambda_4),\\
\lambda_1+\lambda_2+\lambda_3+\lambda_4=1,\;\; \lambda_i\geq 0,\;\; i=1,\ldots,4\,.
\end{array}
\right.
\end{equation}
\emph{(2\textsuperscript{nd}) Step 1.\,} We set $z=0$ in System (\ref{S_1'})
and obtain
\begin{equation}\label{S_1'z=0}
\left\{
\begin{array}{ll}
\lambda_2=y, \;\;
\lambda_1+\lambda_3=0, \;\;
\lambda_4=1-y,\;\;
\lambda_i\geq 0,\;\; i=1,\ldots,4\,.
\end{array}
\right.
\end{equation}
As $y \in [0,1]$,  System (\ref{S_1'z=0})  is always solvable. Thus, we go to the following:  \\
(2\textsuperscript{nd}) Step 3. 
We denote by $\Lambda$  and $\mathcal{S}$ the vector of unknowns $(\lambda_1,\ldots,\lambda_4)$ and the set of solution of System~(\ref{S_1'z=0}), respectively.  By recalling the conditioning events $H_2=A$ and $H_4=A$, we consider the following linear functions $\Phi_i(\Lambda)$:
$\Phi_2(\Lambda)=\Phi_4(\Lambda)=\sum_{r:C_r\subseteq A}\lambda_r=\lambda_1+\lambda_2+\lambda_3+\lambda_4$.
We set  $M_i=\max_{\Lambda \in \mathcal{S}} \Phi_i(\Lambda)$,  $i=2,4$. From System~(\ref{S_1'z=0}), we obtain: $\Phi_2(\Lambda)=\Phi_4(\Lambda)=1,\;\; \forall \Lambda \in \mathcal{S}$. Then, $M_2=M_4=1$.
We are in case 1 of Step 3 of the algorithm; then the \emph{procedure stops} and yields   $z'=0$ as \emph{output}. 

To summarize,  for any  $(x,y,t)\in[0,1]^3$ on $(C|B,B|A,A|(A\vee B))$, we have computed the  coherent lower bound $z'$ on $C|A$. In particular,  if $t=0$, then $z'=0$. Moreover, if $t>0$ and $t(1-x+xy)\leq 1-x$, that is $xy-(1-t)(1-x)/t\leq 0$, we also have  $z'=0$. Finally, if $t(1-x+xy)>1-x$, then ${z'=xy-(1-t)(1-x)/t}$.
\paragraph{Computation of the upper  probability bound  $z''$ on $C|A$.}~\\
\emph{Input} and \emph{Step 0} are  the same as  in the proof of $z'$.\\ 
\emph{Step 1.\,} We set $z=1$ in  System (\ref{S_0'}) and  obtain
\begin{equation}\label{S_0'z=1}
\begin{array}{lll}
\left\{
\begin{array}{ll}
\lambda_2+\lambda_4=0,  \\
\lambda_1+\lambda_5=x(\lambda_1+\lambda_5+\lambda_6), \\
\lambda_1=y(\lambda_1+\lambda_3),\;\;
\lambda_1+\lambda_3=t,  \\
\lambda_1+\lambda_3+\lambda_5+\lambda_6=1, \\
\lambda_i\geq 0,\; i=1,\ldots,6\,,
\end{array}
\right.
& \quad \Longleftrightarrow \quad &
\left\{
\begin{array}{lllllll}
\lambda_1=yt, \;\; \lambda_3=t(1-y), \\
\lambda_2=\lambda_4=0,\\
\lambda_5=x-xt+xyt-yt, \\
\lambda_6=(1-x)[1- t(1-y)], \\
\lambda_i\geq 0,\; i=1,\ldots,6\,.
\end{array}
\right.
\end{array}
\end{equation}
As $(x,y,t) \in [0,1]^3$, we observe that: $yt\geq 0$, $ t(1-y)\geq 0$, and $(1-x)[1-t(1-y)]\geq 0$.  
Thus, System (\ref{S_0'z=1})   is solvable  if and only if $
x+xyt-xt-yt\geq 0$, i.e., $t(x+y-xy)\leq x$.
We distinguish two cases: $(i)$ $x+xyt-xt-yt< 0$; $(ii)$ $x+xyt-xt-yt\geq 0$. In Case~$(i)$,  System~(\ref{S_0'z=1})   is not solvable and---according to the algorithm---we proceed to  Step~2. In Case~$(ii)$,  System~(\ref{S_0'z=1})  is solvable and we go to Step~3. 

\emph{Case $(i)$.\,} We take Step 2 and consider the following linear programming problem:
$
\mbox{\em Compute }z''=\max(  \lambda_1+\lambda_3),\;
\mbox{\em subject to the constraints in } (\ref{Sigma'}).
$
As $x+xyt-xt-yt< 0$, that is $t(x+y-xy)>x\geq 0$,  it holds that $t>0$. In this case,  the constraints in (\ref{Sigma'}) can be rewritten as in (\ref{Sigma'1}). 
Since $x+xyt-xt-yt< 0$, it holds that $x+xyt-xt<yt\leq y$. Thus, 
we obtain the maximum of $(\lambda_1+\lambda_3)$
subject to~(\ref{Sigma'1})  at $(\lambda_1^{''},\lambda_3^{''})=\left(
xy-x+x/t,1-y\right)$.
The \emph{procedure stops} and yields the following  \emph{output}: $z''= 1-y-x+xy+x/t=(1-x)(1-y)+x/t$.

\emph{Case $(ii)$.\,} We take Step~3 of the algorithm.
We denote by $\Lambda$  and $\mathcal{S}$ the vector of unknowns $(\lambda_1,\ldots,\lambda_6)$ and the set of solution of System~(\ref{S_0'z=1}), respectively. We consider the  functions  given in (\ref{EQ:PHI}). From System~(\ref{S_0'z=1}), we  obtain  $M_1=yt+1-t$,  $M_2=M_4=t$, and $M_3=1$. If $t>0$, then $M_{4}>0$ and we are in case 1 of Step 3. Thus,  the \emph{procedure stops} and yields $z''=1$ as \emph{output}.  If $t=0$, then
$M_{1}>0,M_{3}>0$ and $M_{2}=M_{4}=0$. Hence, we are in case 3 of Step 3 with $J = \{2\}, I_0=\{2,4\}$ and the procedure restarts with Step 0, with $\F_n$ replaced by  $\F_J=(E_2|H_2)=(B|A)$ and $\A_n$ replaced by $\A_J=([\alpha_2,\beta_2])=([y,y])$.\\
\emph{(2\textsuperscript{nd}) Step 0.\,} This is the same as the (2\textsuperscript{nd})  Step 0 in the proof of $z'$. \\
\emph{(2\textsuperscript{nd}) Step 1.\,} We set $z=1$ in  System (\ref{S_0'}) and  obtain
\begin{equation}\label{S_1'z=1}
\left\{
\begin{array}{ll}
\lambda_1=y,\;\;
\lambda_3=1-y, \;\;
\lambda_2+\lambda_4=0, \;\;
\lambda_i\geq 0,\;\; i=1,\ldots,4\,.
\end{array}
\right.
\end{equation}
As $y \in [0,1]$,  System (\ref{S_1'z=1})  is always solvable. Thus, we go to the following:  \\
\emph{(2\textsuperscript{nd}) Step 3.} Like in  (the 2\textsuperscript{nd})  Step 3 of the proof of $z'$,
we obtain $M_4=1$. Thus, the procedure stops and yields   $z''=1$ as output.

To summarize,  for any  assessment  $(x,y,t)\in[0,1]^3$ on $(C|B,B|A,A|(A\vee B))$, we have computed the  coherent upper probability bound $z''$ on $C|A$.  In particular,  if $t=0$, then $z''=1$. Moreover, if $t>0$ and $t(x+y-xy)\leq x$, that is $(x+y-xy)\leq \frac{x}{t}$ $\Longleftrightarrow$ $\frac{x}{t}-x-y+xy\geq 0$ $\Longleftrightarrow$
$(1-x)(1-y)+\frac{x}{t}\geq 1$, we also have  $z''=1$. Finally, if $t(x+y-xy)>x$, then $z''=(1-x)(1-y)+\frac{x}{t}$. This concludes the proof of  the following:
\begin{theorem}\label{THM:PROPWT}
Let $A,B,C$ be three logically independent events and $(x,y,t)\in[0,1]^3$ be a (coherent)  assessment on the family $\big(C|B,B|A,A|(A\vee B)\big)$. Then, the  extension $z=P(C|A)$ is coherent if and only if $z\in[z',z'']$, where 
\[
\begin{array}{ll}
[z',z''] =
\left\{
\begin{array}{ll}
[0,1], & t=0;\\
\left[\max\{0,xy-(1-t)(1-x)/t\}, \min\{1,(1-x)(1-y)+x/t\}
\right] \,, & t>0\,.\\
\end{array}
\right.
\end{array}
\]
\end{theorem}
\section{Weak transitivity involving (negated) defaults}
\label{SEC:WTIND}
By Remark \ref{REM:TOTCOHWT}, 
the p-consistent knowledge base ${(B\normally C, A\normally B)}$ neither  p-entails $A\normally C$ nor p-entails $A\nnormally C$. This will be denoted by ${(B\normally C, A\normally B) \nvDash_p A\normally C}$ and $(B\normally C, A\normally B) \nvDash_p A\nnormally C$, respectively. 
\begin{theorem} \label{THM:WT1}
$(B\normally C, A\normally B, A\vee B \nnormally \neg A) \models_p A\normally C$.
\end{theorem}
\begin{proof}
By   Remark~\ref{REM:TOTCOHWT}, the knowledge base $\mathcal{K}=(B\normally C, A\normally B, A\vee B \nnormally \neg A)$ is p-consistent. 
Based on  Remark~\ref{REM:CONJUGACY},
we set $\mathcal{I}_{\mathcal{K}}=\{1\}\times\{1\}\times{]0,1]}$ and  $\mathcal{F}_{\mathcal{K}}=\big (C|B,B|A,A|(A\vee B)\big)$.  Let $\mathcal{P}$ be any precise coherent  assessment on $\mathcal{F}_{\mathcal{K}}$ such that $\mathcal{P}\in \mathcal{I}_{\mathcal{K}}$, i.e., $\mathcal{P}=(1,1,t)$, with $t\in]0,1]$. From Theorem~\ref{THM:PROPWT},  the interval of  coherent extensions from $\mathcal{P}$  on $\mathcal{F}_{\mathcal{K}}$ to $C|A$ is $[z_{\mathcal{P}}',z_{\mathcal{P}}'']=[1,1]$. Then, by Equation (\ref{EQ:UEXT}), the set of coherent extensions to $C|A$ from $\mathcal{I}_{\mathcal{K}}$ on $\mathcal{F}_{\mathcal{K}}$ is
$
\bigcup_{\mathcal{P} \in \mathcal{I}_{\mathcal{K}}}  [z_{\mathcal{P}}',z_{\mathcal{P}}'']=[1,1]\,
$\,.
\qed
\end{proof}
\begin{theorem}\label{THM:AII}
$(B\normally C, A\nnormally \neg B, A\vee B \nnormally \neg A) \models_p A\nnormally \neg C$.
\end{theorem}
\begin{proof}
By   Remark~\ref{REM:TOTCOHWT}, the knowledge base $\mathcal{K}=(B\normally C, A\nnormally \neg B, A\vee B \nnormally \neg A)$ is p-consistent. 
Based on  Remark~\ref{REM:CONJUGACY},
we set $\mathcal{I}_{\mathcal{K}}=\{1\}\times{]0,1]}\times{]0,1]}$ and  $\mathcal{F}_{\mathcal{K}}=\big (C|B,B|A,A|(A\vee B)\big)$.  Let $\mathcal{P}$ be any precise coherent  assessment on $\mathcal{F}_{\mathcal{K}}$ such that $\mathcal{P}\in \mathcal{I}_{\mathcal{K}}$, i.e., $\mathcal{P}=(1,y,t)$, with $y\in]0,1]$ and  $t\in]0,1]$. From Theorem~\ref{THM:PROPWT},  the interval of  coherent extensions from $\mathcal{P}$  on $\mathcal{F}_{\mathcal{K}}$ to $C|A$ is $[z_{\mathcal{P}}',z_{\mathcal{P}}'']=[y,1]$. Then, by Equation (\ref{EQ:UEXT}), the set of coherent extensions to $C|A$ from $\mathcal{I}_{\mathcal{K}}$ on $\mathcal{F}_{\mathcal{K}}$ is
$
\bigcup_{\mathcal{P} \in \mathcal{I}_{\mathcal{K}}}  [z_{\mathcal{P}}',z_{\mathcal{P}}'']=
\bigcup_{(y,t)\in]0,1]\times]0,1]} [y,1]=
]0,1]\,
$\,. Therefore,  the set of coherent extensions on $\neg C|A$ is $[0,1[$.
\qed
\end{proof}
\begin{theorem}\label{THM:WT4}
 $(B\normally C, A\normally  B,  B \nnormally \neg A) \models_p A\normally C$.
\end{theorem}
\begin{proof}
It can be shown that the assessment $[0,1]^3$ on $(C|B,B|A,A|B)$ is totally coherent. Then, $\mathcal{K}=(B\normally C, A\normally  B,  B \nnormally \neg A)$ is p-consistent. 
We set $\mathcal{I}_{\mathcal{K}}=\{1\}\times\{1\}\times{]0,1]}$ and  $\mathcal{F}_{\mathcal{K}}=\big (C|B,B|A,A|B\big)$.
We observe that  $A|B \subseteq  A|(A\vee B)$, where the binary relation $\subseteq$ denotes the well-known Goodman and Nguyen inclusion relation between conditional events (e.g., \cite{gilio13}). 
Coherence requires that $p(A|B)\leq p(A|(A \vee B))$.
Let $\mathcal{P}$ be any precise coherent  assessment on $\mathcal{F}_{\mathcal{K}}$ such that $\mathcal{P}\in \mathcal{I}_{\mathcal{K}}$, i.e., $\mathcal{P}=(1,1,w)$, with  $w\in\;]0,1]$.  Thus, for any coherent extension  $\mathcal{P}'=(1,1,w,t)$ of $\mathcal{P}$ on $(\mathcal{F}_{\mathcal{K}},A|(A \vee B))$,  it holds that $0<w\leq t$.  Then, $\mathcal{K}'=(B\normally C, A\normally  B,  B \nnormally \neg A, A\vee B \nnormally \neg A)$ is p-consistent. Thus,  by  Theorem \ref{THM:WT1}, $\mathcal{K}' \models_p A\normally C$. Then,  for every coherent extension $\mathcal{P}''=(1,1,w,t,z)$ of $\mathcal{P}'$ on $(\mathcal{F}_{\mathcal{K'}},C|A)$ it holds  that $z=1$.
By \emph{reductio ad absurdum}, if for some $z<1$ the extension $(1,1,w,z)$  on  $(\mathcal{F}_{\mathcal{K}},C|A)$ of $\mathcal{P}\in \mathcal{I}_{\mathcal{K}}$ on $\mathcal{F}_{\mathcal{K}}$ were coherent, then---with  $0<w\leq t$ and $z<1$---the assessment $(1,1,w,t,z)$ on $(\mathcal{F}_{\mathcal{K'}},C|A)$ would be coherent, which contradicts the conclusion $z=1$ above. Thus, for every  coherent extension $(1,1,w,z)$  of  $\mathcal{P}\in \mathcal{I}_{\mathcal{K}}$ on  $(\mathcal{F}_{\mathcal{K}},C|A)$ it holds that $z=1$.\qed
\end{proof}
\begin{theorem}\label{THM:WT5}
$(B\normally C, A\nnormally \neg B,  B \nnormally \neg A) \models_p A\nnormally \neg C$.
\end{theorem}
\begin{proof}
The proof exploits Theorem~\ref{THM:AII} and is similar to the proof of Theorem~\ref{THM:WT4}.
\end{proof}
\section{Concluding remarks}\vspace{-1em}
Our definition of negated defaults, based on imprecise probabilities (Sect.~\ref{SEC:ProbKBENT}), can be
seen as an instance of the \emph{wide-scope} reading of the negation
of a conditional. It offers an interesting alternative to the
\emph{narrow-scope} reading, where a conditional is negated by
negating its consequent \cite{pfeifer12x}. Moreover, we note that
Theorem~\ref{THM:WT1} can be seen as a modern formalization of the
classical (Aristotelian) syllogistic Modus Barbara (with ${A\vee B
\nnormally \neg A}$, i.e. $P(A|(A\vee B))>0$, as an existential import assumption). Likewise,
Theorem~\ref{THM:AII} as Modus Darii. A different (stronger) existential import
assumption ($B\nnormally \neg A$, i.e. $P(A|B)>0$) for Modus Barbara and Modus Darii is considered in
theorems~\ref{THM:WT4} and~\ref{THM:WT5}, respectively. We are
currently working on a coherence-based probability
semantics for classical syllogisms, where we exploit ideas presented above.

\begin{thebibliography}{10}
\providecommand{\url}[1]{\texttt{#1}}
\providecommand{\urlprefix}{URL }

\bibitem{capotorti09}
Baioletti, M., Capotorti, A., Galli, L., Tognoloni, S., Rossi, F., Vantaggi,
  B.: Ck{C}-package; version e5.
  \url{www.dmi.unipg.it/~upkd/paid/software.html} (2009)

\bibitem{benferhat97}
Benferhat, S., Dubois, D., Prade, H.: Nonmonotonic reasoning, conditional
  objects and possibility theory. Artificial Intelligence  92,  259--276 (1997)

\bibitem{biazzo00}
Biazzo, V., Gilio, A.: A generalization of the fundamental theorem of de
  {F}inetti for imprecise conditional probability assessments. International
  Journal of Approximate Reasoning  24(2-3),  251--272 (2000)

\bibitem{biazzo02}
Biazzo, V., Gilio, A., Lukasiewicz, T., Sanfilippo, G.: Probabilistic logic
  under coherence, model-theoretic probabilistic logic, and default reasoning
  in {S}ystem {P}. Journal of Applied Non-Classical Logics  12(2),  189--213
  (2002)

\bibitem{bonnefon12}
Bonnefon, J.F., Da~Silva~Neves, R., Dubois, D., Prade, H.: Qualitative and
  quantitative conditions for the transitivity of perceived causation:. Annals
  of Mathematics and Artificial Intelligence  64(2-3),  311--333 (2012)

\bibitem{coletti02}
Coletti, G., Scozzafava, R.: Probabilistic logic in a coherent setting. Kluwer,
  Dordrecht (2002)

\bibitem{CoPV14}
Coletti, G., Petturiti, D., Vantaggi, B.: Coherent {T}-conditional possibility
  envelopes and nonmonotonic reasoning. CCIS, vol. 444, pp. 446--455. Springer
  (2014)

\bibitem{dubois93}
Dubois, D., Godo, L., López De~Màntaras, R., Prade, H.: Qualitative reasoning
  with imprecise probabilities. Journal of Intelligent Information Systems
  2(4),  319--363 (1993)

\bibitem{freund1991}
Freund, M., Lehmann, D., Morris, P.: Rationality, transitivity, and
  contraposition. Artificial Intelligence  52(2),  191--203 (1991)

\bibitem{gilio02}
Gilio, A.: Probabilistic reasoning under coherence in {S}ystem {P}. Annals of
  Mathematics and Artificial Intelligence  34,  5--34 (2002)

\bibitem{gilio98}
Gilio, A., Ingrassia, S.: Totally coherent set-valued probability assessments.
  Kybernetika  34(1),  3--15 (1998)

\bibitem{gilio13}
Gilio, A., Sanfilippo, G.: Probabilistic entailment in the setting of
  coherence: The role of quasi conjunction and inclusion relation. IJAR  54(4),
   513--525 (2013)

\bibitem{gilio13ins}
Gilio, A., Sanfilippo, G.: Quasi conjunction, quasi disjunction, t-norms and
  t-conorms: Probabilistic aspects. Information Sciences  245,  146--167 (2013)

\bibitem{GiSa14}
Gilio, A., Sanfilippo, G.: Conditional random quantities and compounds of
  conditionals. Studia Logica  102(4),  709--729 (2014)

\bibitem{pfeifer12x}
Pfeifer, N.: Experiments on {A}ristotle's {T}hesis: {T}owards an experimental
  philosophy of conditionals. The Monist  95(2),  223--240 (2012)

\bibitem{pfeifer13}
Pfeifer, N.: Reasoning about uncertain conditionals. Studia Logica  102(4),
  849--866 (2014)

\bibitem{pfeifer09b}
Pfeifer, N., Kleiter, G.D.: Framing human inference by coherence based
  probability logic. Journal of Applied Logic  7(2),  206--217 (2009)

\bibitem{pfeifer10a}
Pfeifer, N., Kleiter, G.D.: The conditional in mental probability logic. In:
  Oaksford, M., Chater, N. (eds.) Cognition and conditionals, pp. 153--173.
  Oxford Press (2010)

\end{thebibliography}

\end{document}